\documentclass{siamart251216}

\usepackage{amsfonts,amssymb,mathtools}
\usepackage{bm}
\usepackage{enumitem}
\usepackage{microtype}
\usepackage{tikz}
\usetikzlibrary{calc}
\usetikzlibrary{positioning}

\newsiamremark{remark}{Remark}
\crefname{remark}{Remark}{Remarks}
\Crefname{remark}{Remark}{Remarks}

\newcommand{\G}{\mathcal G}
\newcommand{\E}{\mathcal E}
\newcommand{\V}{\mathcal V}
\newcommand{\T}{\mathbb T}
\newcommand{\R}{\mathbb R}
\newcommand{\C}{\mathbb C}
\newcommand{\Z}{\mathbb Z}

\newcommand{\supp}{\operatorname{supp}}
\newcommand{\spec}{\operatorname{spec}}
\newcommand{\dd}{\,\mathrm d}

\title{Arithmetic Scar Accessibility and Generic Schr\"odinger Control on Metric Graphs}

\author{Binh T. Nguyen\thanks{Faculty of Mathematics and Computer Science, University of Science, Vietnam National University Ho Chi Minh City, Ho Chi Minh City, 700000, Vietnam (\email{ngtbinh@hcmus.edu.vn}).}}

\headers{Arithmetic Scar Accessibility on Metric Graphs}{Binh T. Nguyen}

\begin{document}
\maketitle

\begin{abstract}
We study internal exact controllability of the free Schr\"odinger equation on
finite compact metric graphs with Kirchhoff conditions at interior vertices and
a fixed assignment of Dirichlet or Neumann conditions at exterior vertices.
Our main result proves that, for every rationally independent metric, the Graph
Geometric Control Condition (GGCC) is necessary and sufficient for exact
controllability in every positive time.  Since rationally dependent length
vectors form a Lebesgue-null set, this yields a generic metric characterization
of exact controllability.

The necessity mechanism is spectral and arithmetic.  For a prescribed
primitive cycle or exterior-to-exterior path, we introduce a fixed-metric
accessibility criterion obtained by intersecting the regular scar stratum of
the secular set with the orbit-closure torus of the target length vector.
Whenever this intersection is nonempty, the prescribed scar is realized
asymptotically by exact high-frequency eigenfunctions of the target metric.
The criterion also applies to certain resonant metrics beyond the rationally
independent class.  Under an intrinsic Diophantine condition on the
orbit-closure flow, the construction is quantitative: the mass outside the
prescribed support decays polynomially along an exact eigensequence, which
yields a polynomial degeneration law for finite-frequency observation whenever
that support is uncontrolled.  Rationally independent metrics have full
phase-orbit closure, making every primitive obstruction accessible; combined
with the known graph-theoretic characterization and sufficiency of GGCC, this
gives the stated controllability characterization.
\end{abstract}



\section{Introduction}
\label{sec:intro}

The Graph Geometric Control Condition (GGCC), introduced by Ammari and colleagues ~\cite{AmmariDucaJolyLeBalch2025}, gives a purely
graph-theoretic condition ensuring internal exact controllability of the
Schr\"odinger equation on a finite metric graph.  For the free equation,
however, GGCC is not necessary at every metric: arithmetic exceptional length
vectors may remain controllable even when an uncontrolled primitive subgraph is
present.  This led to Conjecture~7.7 of~\cite{AmmariDucaJolyLeBalch2025}, which
asks whether GGCC becomes necessary, and hence characterizes controllability,
for generic edge lengths.  We prove this generic converse for the free
Laplacian, with arbitrary fixed Dirichlet--Neumann conditions at the exterior
vertices.

The difficulty is that failure of GGCC is topological, whereas failure of
Schr\"odinger observability must be witnessed spectrally.  The graph-theoretic
characterization of GGCC produces an uncontrolled simple cycle or an
uncontrolled simple path joining two exterior vertices.  To turn such a
primitive obstruction into nonobservability, one must construct \emph{exact}
high-frequency eigenfunctions of the \emph{fixed target metric} whose mass
concentrates on that prescribed subgraph.  Colin de Verdi`ere's
determinant-manifold theory~\cite{CdV2015} gives the classical generic
minimal-support picture for standard Kirchhoff graphs, including simple cycles
and simple paths between degree-one vertices. Accordingly, the cycle/path classification belongs to the established spectral
theory.  The problem addressed here is different: for a fixed metric and a
prescribed primitive support, we ask whether the corresponding regular scar is
accessible by the exact spectrum of that same metric.

Our first contribution is a fixed-metric arithmetic accessibility criterion.
For a target length vector \(\ell\), let \(H_\ell\) denote the closure of the
phase flow \(k\mapsto[k\ell]\), and for a primitive support \(K\), let
\(\mathcal R_K\) denote the regular secular phases carrying a simple
\(K\)-supported eigendirection.  We prove that a regular scar phase can be
reached by exact target spectral points whenever it belongs to the orbit
closure.  This motivates the regular arithmetic accessibility condition
\[
H_\ell\cap\mathcal R_K\neq\varnothing.
\]
It is metric-specific and obstruction-specific: it may hold for resonant
metrics, while rational independence is only the special case
\(H_\ell=\T^\E\), in which every regular primitive scar is accessible.

Our second contribution is quantitative.  Under an intrinsic Diophantine
condition on the linear flow in \(H_\ell\), quantitative torus filling and the
local secular geometry yield exact spectral crossings approaching a prescribed
regular scar at a polynomial rate.  Smoothness of the corresponding simple
secular eigenbranch gives a quadratic leakage law, so the off-support
\(L^2\)-mass decays polynomially along the resulting exact eigensequence.  For
an uncontrolled primitive support, this produces a polynomial degeneration of
the finite-frequency observation level.

Our third contribution is the control-theoretic consequence.  Rationally
independent metrics have full phase-orbit closure and hence make every primitive
obstruction arising from failure of GGCC regularly accessible.  The resulting
scarred eigensequence violates observability at every positive time.  Combined
with the GGCC sufficiency theorem of~\cite{AmmariDucaJolyLeBalch2025}, this
proves that GGCC is necessary and sufficient for exact controllability for every
rationally independent metric.  Since rationally dependent positive length
vectors form a Lebesgue-null set, the same characterization consequently holds
for Lebesgue-almost every positive length vector.  In this way, the fixed-metric
scar theorem and its quantitative Diophantine refinement provide the spectral
mechanism underlying the generic GGCC characterization.

\subsection{Main results}
\label{subsec:main-results}

Throughout, \(\G\) is a finite compact connected metric graph with Kirchhoff
conditions at interior vertices and a fixed Dirichlet--Neumann assignment at
exterior vertices.  The control set \(\omega\) is a nonempty open subset of the
disjoint union of edge interiors.  The definitions of \(H_\ell\),
\(\mathcal R_K\), and the regular arithmetic accessibility condition
\(\operatorname{RAC}(K,\ell)\) are given in 
\Cref{sec:prelim,sec:secular-transfer}.

Our first main result is a fixed-metric scar theorem, as follows.
\begin{theorem}[Arithmetic primitive-scar accessibility]
\label{thm:intro-arithmetic-scar}
Let \(K\subset\G\) be a proper simple cycle or a proper simple path joining two
exterior vertices.  If \(\operatorname{RAC}(K,\ell)\) holds, then there exist
normalized exact eigenfunctions \(u_n\) and frequencies \(k_n\to\infty\) such
that
\begin{equation}
\label{eq:intro-arithmetic-scar}
-\Delta_{\G,\ell}u_n=k_n^2u_n,
\qquad
\|u_n\|_{L^2(\G\setminus K)}\longrightarrow0,
\qquad
|u_n|^2\dd x\rightharpoonup\frac1{L_K}\,\dd x\big|_K.
\end{equation}
For exterior paths, all four endpoint types \(DD,DN,ND,NN\) are allowed.
\end{theorem}

The condition \(\operatorname{RAC}(K,\ell)\) is equivalent to the existence of
a regular \(K\)-scar phase satisfying every integer resonance relation of
\(\ell\); see \cref{lem:orbit-closure-lattice,def:RAC}.  In particular,
rationally independent metrics satisfy \(\operatorname{RAC}(K,\ell)\) for
every proper primitive \(K\).

The second result quantifies \eqref{eq:intro-arithmetic-scar}.  Let
\(r_\ell=\dim H_\ell\).  Under the intrinsic Diophantine hypothesis stated in
\cref{def:intrinsic-diophantine}, we obtain a sequence in
\cref{thm:intro-arithmetic-scar} satisfying
\begin{equation}
\label{eq:intro-quantitative-scar}
\|u_n\|_{L^2(\G\setminus K)}^2
\le C k_n^{-2/\tau}.
\end{equation}
Estimate \eqref{eq:intro-quantitative-scar} is the quantitative conclusion of
\cref{thm:quantitative-scar}; the periodic case \(r_\ell=1\) is treated
separately in \cref{thm:quantitative-crossing}.

To express the control consequence, define the truncated observation level
\(\mathfrak o_T(\Lambda)\) in \eqref{eq:truncated-observation-level}.  If GGCC
fails and an uncontrolled primitive \(K\) satisfies the quantitative arithmetic
hypotheses, then along \(\Lambda_n=k_n^2\),
\begin{equation}
\label{eq:intro-observation-degeneration}
\mathfrak o_T(\Lambda_n)
\le C T\Lambda_n^{-1/\tau}.
\end{equation}
Estimate \eqref{eq:intro-observation-degeneration} shows that the inverse finite-frequency observation level grows at least
polynomially along a spectral subsequence.

Finally, rational independence makes \(H_\ell=\T^\E\), so every proper
primitive obstruction is regularly accessible.  Combining
\cref{thm:intro-arithmetic-scar} with the GGCC topology and the sufficiency
theorem of Ammari--Duca--Joly--Le Balc'h yields the generic control
characterization.

\begin{theorem}[Generic GGCC characterization]
\label{thm:intro-control}
If \(\ell\in(0,\infty)^\E\) is rationally independent, then
\((\G,\omega)\) satisfies GGCC if and only if the free internally controlled
Schr\"odinger equation is exactly controllable in every time \(T>0\).  Hence, the
same equivalence holds for Lebesgue-almost every positive length vector.
\end{theorem}

\subsection{Related Work}
\label{subsec:related-work}

\paragraph{Graph control and the GGCC problem}
Control and observation of PDEs on metric graphs have been studied through
boundary, vertex, and distributed control mechanisms.  For wave equations,
representative recent results include arithmetic controllability on star
networks~\cite{Dager2016}, constructive and inverse formulations on metric
graphs~\cite{Avdonin2019IFAC}, exact controllability on trees and more general
graphs~\cite{AvdoninEdwardLeugering2023,AvdoninZhao2021,AvdoninZhao2022,AvdoninEdward2025}.  For Schr\"odinger dynamics, bilinear controllability on
compact and infinite graphs was developed in
\cite{AmmariDuca2020JMP,AmmariDuca2021IJC,Duca2020SICON,Duca2021Automatica};
related quantum-control models on graph-like manifolds appear in
\cite{BalmasedaLonigroPerezPardo2023}.  These works provide important control
mechanisms but do not characterize internal controllability by a graph-geometric
condition for the free Schr\"odinger flow.  Ammari et al.
\cite{AmmariDucaJolyLeBalch2025} introduced GGCC, proved its graph-theoretic
characterization and sufficiency for Schr\"odinger controllability, exhibited
arithmetic exceptional metrics showing that necessity fails pointwise, and
proved a converse in a rationally independent Dirichlet-star setting.  Their
Conjecture~7.7 asks for the generic converse on arbitrary finite graphs.  Our
control theorem proves that converse for the free Laplacian with arbitrary
fixed exterior Dirichlet--Neumann assignments.  The proof also identifies a
metric-specific regular accessibility condition that explains why generic
metrics and exceptional arithmetic metrics need not behave alike.

\paragraph{Scarring and fixed-metric accessibility}
Colin de Verdi`ere~\cite{CdV2015} relates semiclassical measures on quantum
graphs to the Gauss map of the determinant manifold and, for generic metrics in
the standard Kirchhoff setting, identifies simple cycles and simple paths
between degree-one vertices as the minimal supports.  This provides the main spectral background for our analysis, while the present
work concerns the distinct fixed-metric problem of determining whether a
prescribed primitive scar is accessible by the exact spectrum.
 Earlier results established failure
of quantum ergodicity on star graphs~\cite{BerkolaikoKeatingWinn2004}, while
Berkolaiko and Winn~\cite{BerkolaikoWinn2018} constructed maximally scarred
nondegenerate states in star-graph settings.  In the opposite direction,
delocalization and full-support phenomena occur in other
regimes~\cite{IngremeauSabriWinn2020,PlumerTaufer2021}.

The distinction relevant to the present work is between \emph{existence or
classification of scar supports} and \emph{accessibility of a prescribed
regular scar by the exact spectrum of a fixed target metric}.  We address the
latter through regular arithmetic accessibility (RAC), formulated by the
intersection of the regular scar stratum with the orbit-closure torus of the
target metric.  This formulation also applies to resonant metrics for which the
phase orbit is confined to a proper subtorus.

\paragraph{Secular geometry and scar-preserving regularization}
Bond scattering and secular manifolds are standard tools in quantum-graph
spectral theory; see~\cite{BerkolaikoKuchment2013}.  Generic simplicity and
nonvanishing under edge-length perturbations were established by
Berkolaiko and Liu~\cite{BerkolaikoLiu2017}, with further generic eigenfunction
results in~\cite{Alon2024Generic}; related nodal and secular structures appear
in~\cite{AlonBandBerkolaiko2018,KurasovMuller2021}.  The regularization used in
this paper has a different constraint: a prescribed scar eigenpair is retained
exactly while only off-support lengths are perturbed to obtain a regular ambient
secular point.  This constrained step is what permits subsequent transfer back
to the original fixed target metric.

\paragraph{Quantitative observation and recurrence}
Spectral inequalities provide a complementary route to observability.
Egidi et al.~\cite{EgidiMugnoloSeelmann2024} prove
Logvinenko--Sereda-type estimates for Sturm--Liouville systems with applications
to quantum graphs and suitably distributed observation sets.  Our case is the
opposite geometric one: when GGCC fails, an entire primitive cycle or
exterior-to-exterior path may avoid the observation region, and the issue is how
rapidly exact eigenfunctions can concentrate on that uncontrolled subgraph.
For this purpose, we combine the local secular analysis with the quantitative
torus-filling estimates of Dumas and Fischler~\cite{DumasFischler2022}.  Rational
independence supplies qualitative density only; the polynomial scar and
observation rates require the stronger intrinsic Diophantine hypothesis used in
\Cref{sec:quantitative}.

\paragraph{Positioning of the present work}
Relative to the closest literature, the present contribution has three
distinct components.  First, rather than revisiting the classification of
minimal semiclassical supports, we study the fixed-metric accessibility of a
prescribed regular primitive scar.  Second, the argument goes beyond generic
spectral simplicity by constructing a scar-preserving regularization and then
using arithmetic recurrence to return to the original target metric.  Third,
the spectral construction admits a quantitative refinement: under an intrinsic
Diophantine condition, it yields polynomial off-support leakage and a
corresponding degeneration of finite-frequency observation.  For rationally
independent metrics, the same accessibility mechanism applies to every
primitive obstruction arising from failure of GGCC and leads to the generic
necessity result for Schr\"odinger control.

\section{Metric-graph, semiclassical, and arithmetic preliminaries}
\label{sec:prelim}

Let \(\G=(\V,\E,\ell)\) be a finite connected compact metric graph, with every
edge \(e\in\E\) identified with an interval \([0,\ell_e]\), where
\(\ell=(\ell_e)_{e\in\E}\in(0,\infty)^\E\).  A vertex is called
\emph{exterior} if it has degree one and \emph{interior}, otherwise.  Loops count
twice toward vertex degree.  Loops and parallel edges are allowed; by convention,
a loop is a one-edge simple cycle and two parallel edges may form a two-edge
simple cycle.  The control region \(\omega\) is open in the disjoint union of
edge interiors.  The state space is defined by
\begin{equation}
\label{eq:graph-L2-space}
L^2(\G)=\bigoplus_{e\in\E}L^2(0,\ell_e).
\end{equation}
We use the decomposition in \eqref{eq:graph-L2-space} throughout.
At every interior vertex \(v\), functions in the operator domain satisfy
\begin{equation}
\label{eq:kirchhoff-conditions}
u_e(v)=u_{e'}(v)
\quad\text{for all }e,e'\sim v,
\qquad
\sum_{e\sim v}\partial_\nu u_e(v)=0.
\end{equation}
At every exterior vertex, we fix once and for all, either the Dirichlet condition
\(u(v)=0\) or the Neumann condition \(\partial_\nu u(v)=0\).  Together with \eqref{eq:kirchhoff-conditions}, these
conditions define the self-adjoint Laplacian \(-\Delta_{\G,\ell}\).

A \emph{semiclassical measure} for the fixed metric \(\ell\) is a weak limit of
probability measures \(|u_n|^2\dd x\), where
\(-\Delta_{\G,\ell}u_n=k_n^2u_n\), \(\|u_n\|_{L^2(\G)}=1\), and
\(k_n\to\infty\).  A support is called \emph{minimal} if it is minimal under
set inclusion among supports of such probability semiclassical measures.

Set \(\T=\R/(2\pi\Z)\).  The phase-orbit closure and resonance lattice of a
positive target metric are
\begin{equation}
\label{eq:orbit-closure}
H_\ell:=\overline{\{[k\ell]:k\in\R\}}\subset\T^\E
\end{equation}
and
\begin{equation}
\label{eq:resonance-lattice}
\Lambda_\ell:=\{m\in\Z^\E:m\cdot\ell=0\}.
\end{equation}
The orbit closure in \eqref{eq:orbit-closure} and the resonance lattice in \eqref{eq:resonance-lattice} are related by the standard Kronecker description:
\begin{lemma}[Orbit-closure lattice criterion]
\label{lem:orbit-closure-lattice}
For every positive \(\ell\),
\begin{equation}
\label{eq:orbit-closure-lattice}
H_\ell
=
\left\{\theta\in\T^\E:
 m\cdot\theta\in2\pi\Z
 \text{ for every }m\in\Lambda_\ell
\right\}.
\end{equation}
\end{lemma}

\begin{proof}
The annihilator in \(\Z^\E\) of the one-parameter subgroup
\(\{[k\ell]:k\in\R\}\) is precisely \(\Lambda_\ell\).  A closed subgroup of a
torus is the common zero set of its annihilator characters.  This gives
\eqref{eq:orbit-closure-lattice}.
\end{proof}

\begin{definition}[Rationally independent metric]
\label{def:irrational}
A positive length vector \(\ell\in(0,\infty)^\E\) is rationally independent if
\(m\cdot\ell=0\), \(m\in\Z^\E\), implies \(m=0\).
\end{definition}

By \eqref{eq:resonance-lattice}--\eqref{eq:orbit-closure-lattice}, rational
independence is equivalent to \(H_\ell=\T^\E\).

\begin{remark}
\label{rem:no-quotient}
Degree-two Kirchhoff vertices may be suppressed for local ODE descriptions of
model states.  No quotient graph is used in the target-metric recurrence
argument.  The phase variables, \(H_\ell\), \(\Lambda_\ell\), and all secular
dynamics are formulated on the original edge set \(\E\).
\end{remark}

\section{Primitive supports and boundary-compatible model scars}
\label{sec:model-scars}

We begin with two structural properties of semiclassical measures that hold for
every fixed positive metric and every fixed assignment of Dirichlet or Neumann
conditions at the exterior vertices, as illustrated in
\cref{fig:primitive-supports}.  These properties constrain the possible
supports independently of the arithmetic accessibility mechanism developed
later.

\begin{figure}[t]
\centering
\begin{tikzpicture}[
    scale=0.86,
    transform shape,
    x=1cm,
    y=1cm,
    line cap=round,
    line join=round,
    every node/.style={font=\small},
    interior/.style={circle, fill=black, inner sep=2.2pt},
    exterior/.style={circle, draw=black, fill=white, line width=0.7pt,
                     inner sep=3.0pt}
]

\node[anchor=west] at (-0.2,3.4)
    {\textnormal{(a) Simple cycle (one-edge loop)}};

\coordinate (vloop) at (1.6,0.7);

\draw[line width=0.8pt]
    (vloop)
    .. controls (0.0,0.9) and (0.0,3.0) ..
    (1.6,3.0)
    .. controls (3.2,3.0) and (3.2,0.9) ..
    (vloop);

\node[interior] at (vloop) {};
\node[below=2pt of vloop] {$v$};

\node at (1.6,2.85) {$e$};
\node at (1.6,1.85) {$\ell_e$};

\node[align=left, anchor=north west, text width=4.3cm]
    at (-0.1,0.05)
    {A one-edge simple cycle.\\
     The loop \(e\) contributes two incidences at the vertex \(v\).};

\node[anchor=west] at (5.0,3.4)
    {\textnormal{(b) Simple exterior-to-exterior path}};

\coordinate (v0) at (5.6,2.1);
\coordinate (v1) at (7.4,2.1);
\coordinate (v2) at (9.2,2.1);
\coordinate (vm1) at (11.0,2.1);
\coordinate (vm) at (12.8,2.1);

\draw[line width=0.8pt] (v0) -- (v1);
\draw[line width=0.8pt] (v1) -- (v2);

\draw[line width=0.8pt, dotted]
    (v2) -- (vm1);

\draw[line width=0.8pt] (vm1) -- (vm);

\node[exterior] at (v0) {};
\node[interior] at (v1) {};
\node[interior] at (v2) {};
\node[exterior] at (vm) {};

\node[below=4pt of v0] {$v_0$};
\node[below=4pt of v1] {$v_1$};
\node[below=4pt of v2] {$v_2$};
\node[below=4pt of vm] {$v_m$};

\node[above=3pt] at ($(v0)!0.5!(v1)$) {$e_1$};
\node[above=3pt] at ($(v1)!0.5!(v2)$) {$e_2$};
\node[above=3pt] at ($(vm1)!0.5!(vm)$) {$e_m$};

\node[below=12pt] at ($(v0)!0.5!(v1)$) {$\ell_{e_1}$};
\node[below=12pt] at ($(v1)!0.5!(v2)$) {$\ell_{e_2}$};
\node[below=12pt] at ($(vm1)!0.5!(vm)$) {$\ell_{e_m}$};

\node[align=center, anchor=north]
    at ($(v0)+(0,-0.55)$)
    {\footnotesize exterior\\[-1pt]
     \footnotesize Dirichlet or Neumann};

\node[align=center, anchor=north]
    at ($(v1)+(0,-0.55)$)
    {\footnotesize interior};

\node[align=center, anchor=north]
    at ($(v2)+(0,-0.55)$)
    {\footnotesize interior};

\node[align=center, anchor=north]
    at ($(vm)+(0,-0.55)$)
    {\footnotesize exterior\\[-1pt]
     \footnotesize Dirichlet or Neumann};

\node[align=left, anchor=north west, text width=8.2cm]
    at (5.1,0.05)
    {A simple exterior-to-exterior path
     \(K=\{e_1,\ldots,e_m\}\).  Every interior vertex of the
     support has support degree two.};

\end{tikzpicture}

\caption{Primitive supports on a metric graph. Filled vertices denote
interior vertices carrying the Kirchhoff condition, whereas open vertices
denote exterior vertices with prescribed Dirichlet or Neumann conditions.
A primitive support is either a simple cycle, including a one-edge loop as
shown in panel~(a), or a simple exterior-to-exterior path as shown in
panel~(b).  If a probability semiclassical measure has support exactly equal
to such a primitive support \(K\), then
\(\mu=L_K^{-1}\,\dd x|_K\), where
\(L_K=\sum_{e\subset K}\ell_e\).}
\label{fig:primitive-supports}
\end{figure}

\subsection{Universal support laws}

\begin{lemma}[Edgewise flatness]
\label{lem:edgewise-flatness}
Let \(u_n\) be normalized exact eigenfunctions of \(-\Delta_{\G,\ell}\) with
frequencies \(k_n\to\infty\).  After passage to a subsequence, there are
constants \(\rho_e\ge0\) such that
\begin{equation}
\label{eq:edgewise-flat-limit}
|u_n|^2\dd x\rightharpoonup
\mu:=\sum_{e\in\E}\rho_e\,\dd x\big|_e.
\end{equation}
\end{lemma}

\begin{proof}
On each oriented edge, write
\(u_{n,e}(x)=A_{n,e}e^{ik_nx}+B_{n,e}e^{-ik_nx}\).  Exact integration gives
\[
\int_0^{\ell_e}|u_{n,e}|^2\dd x
=\ell_e\bigl(|A_{n,e}|^2+|B_{n,e}|^2\bigr)+R_{n,e},
\qquad
|R_{n,e}|\le k_n^{-1}
\bigl(|A_{n,e}|^2+|B_{n,e}|^2\bigr).
\]
For large \(n\), positivity of \(\ell_e-k_n^{-1}\) and the global
normalization bound all traveling-wave amplitudes uniformly on the finite edge
set.  Passing to a diagonal subsequence, the two-way intensities converge:
\(|A_{n,e}|^2+|B_{n,e}|^2\to\rho_e\).  For a continuous test function, the
cross term oscillates at frequency \(2k_n\) and tends to zero by the
Riemann--Lebesgue lemma.  Summing over the finite edge set proves
\eqref{eq:edgewise-flat-limit}.
\end{proof}

\begin{lemma}[Interior support law and two-edge balance]
\label{lem:support-law}
Let \(\mu\) be as in \eqref{eq:edgewise-flat-limit} and put
\(H_\mu=\{e\in\E:\rho_e>0\}\).  No interior vertex is a leaf of the subgraph
\(H_\mu\), where incidences are counted with multiplicity and a loop contributes
two incidences.  If exactly two distinct active edge incidences, belonging to edges
\(e,f\in H_\mu\), meet an interior vertex and no other active incidence is
present, then
\begin{equation}
\label{eq:two-edge-balance}
\rho_e=\rho_f.
\end{equation}
\end{lemma}

\begin{proof}
Fix an interior vertex \(v\) of degree \(d\).  In incoming/outgoing bond
coordinates, Kirchhoff matching is
\(a^{\rm out}=S_va^{\rm in}\), where
\(S_v=2J_d/d-I_d\).  After passing, if necessary, to a further subsequence, all finitely many incoming
and outgoing bond amplitudes at the vertices converge.  This does not change the
limits in \eqref{eq:edgewise-flat-limit}.  An inactive edge has two-way intensity
tending to zero and hence both its incoming and outgoing amplitudes tend to zero.

If the active support had degree one at \(v\), exactly one incident bond
coordinate would remain active.  A nonzero one-coordinate incoming limit would be
mapped by \(S_v\) to a vector with nonzero coordinates on the other incident
bonds, a contradiction.  Hence, an interior leaf is impossible.  This formulation
also covers loops correctly because a loop contributes two incidences at its vertex.

Suppose exactly two distinct active incidences, belonging to edges \(e,f\), are
present.  If \(d>2\), the vanishing of all
inactive outgoing limits forces the sum of the two active incoming limits to be
zero; their moduli are therefore equal.  If \(d=2\), the Kirchhoff scattering
matrix swaps the two incoming coordinates.  In either case, the two-way edge
intensities agree, which gives \eqref{eq:two-edge-balance}.
\end{proof}

\begin{corollary}[Uniqueness on a primitive support]
\label{cor:primitive-measure-unique}
If a probability semiclassical measure has support exactly equal to a simple
cycle or a simple exterior-to-exterior path \(K\), then
\begin{equation}
\label{eq:primitive-arclength-measure}
\mu=\frac1{L_K}\,\dd x\big|_K,
\qquad
L_K=\sum_{e\subset K}\ell_e.
\end{equation}
\end{corollary}

\begin{proof}
By \cref{lem:edgewise-flatness}, the measure has a constant density on each
active edge.  Every interior vertex of the primitive support has support-degree two, counting
incidences.  Whenever two distinct consecutive edges meet there,
\cref{lem:support-law} equates their densities; for a one-edge loop, there is
nothing to propagate.  Thus, one common density holds on all edges of \(K\).  Probability normalization determines that density as \(L_K^{-1}\),
which proves \eqref{eq:primitive-arclength-measure}.
\end{proof}

\subsection{Boundary-compatible model scars}

\subsubsection{Exterior-to-exterior path scars}

Let \(\gamma=P(a,b)\subset\G\) be a simple path joining two exterior
vertices \(a\) and \(b\).
The boundary conditions at \(a\) and \(b\) may independently be Dirichlet or Neumann.

Call a vertex \(v\in\gamma\setminus\{a,b\}\) a \emph{branching path vertex}
if at least one edge incident to \(v\) does not belong to \(\gamma\).
These are precisely the vertices at which a path-supported state must vanish,
because the state is zero on every off-path edge and continuity holds at \(v\).

\begin{proposition}[Boundary-compatible path model scar]
\label{prop:path-model-scar}
For every simple exterior-to-exterior path \(\gamma=P(a,b)\) and every fixed
endpoint boundary type \((B_a,B_b)\in\{D,N\}^2\), there exist positive auxiliary edge lengths \(\ell^\star\), a frequency
\(k_\star>0\), and a nonzero eigenfunction
\(\phi_\gamma^\star\) of \(-\Delta_{\G,\ell^\star}\) such that
\[
-\Delta_{\G,\ell^\star}\phi_\gamma^\star
=
k_\star^2\phi_\gamma^\star,
\qquad
\supp\phi_\gamma^\star=\gamma.
\]
\end{proposition}

\begin{proof}
We may fix \(k_\star=1\), since a common rescaling of the auxiliary metric
only rescales the spectral parameter. Let \(v_1,\ldots,v_r\) be the branching path vertices in their path order.

We first treat the case \(r=0\). Since \(\G\) is connected and no vertex of
\(\gamma\) has an off-path attachment, \(\G\) is an interval up to the retained
degree-two subdivisions. Choose the total auxiliary length
\(L_\gamma^\star\) by
\[
L_\gamma^\star\in
\begin{cases}
\pi\mathbb N,&(B_a,B_b)=DD,\\
\pi(\mathbb N+\tfrac12),&(B_a,B_b)=DN\text{ or }ND,\\
\pi\mathbb N,&(B_a,B_b)=NN,
\end{cases}
\]
with a positive integer in the \(NN\) case. The standard sine/cosine mode on
that interval is then a nonzero eigenfunction at eigenvalue \(1\) satisfying
the prescribed separated endpoint conditions. Distribute
\(L_\gamma^\star\) arbitrarily into positive lengths on the constituent
original edges. Degree-two Kirchhoff vertices are exactly \(C^1\) subdivision
points, so the resulting function is an eigenfunction on the original graph.

Assume now that \(r\ge1\). On every off-path edge, define
\(\phi_\gamma^\star\equiv0\). Continuity, therefore, forces
\[
\phi_\gamma^\star(v_j)=0,
\qquad j=1,\ldots,r.
\]
Between consecutive branching vertices, choose the total auxiliary path length
to be an integer multiple of \(\pi\). In such a segment, the solution may be
written, in a local arclength coordinate \(s\), as
\[
\phi_\gamma^\star(s)=A_j\sin(s-s_j),
\]
where \(s_j\) is the coordinate of the left nodal endpoint. At a branching
vertex the off-path normal derivatives vanish, so the Kirchhoff condition
requires cancellation of the two active path derivatives. Consequently, the
amplitude on the next path segment is uniquely determined from the preceding
one, up to the sign forced by the chosen number of half-waves.

For a terminal segment joining a branching node to a Dirichlet exterior
endpoint, choose the total length in \(\pi\mathbb N\); for a Neumann exterior
endpoint, choose it in \(\pi(\mathbb N+\tfrac12)\).
Thus, all four endpoint combinations \(DD,DN,ND,NN\) are compatible. Any
original degree-two Kirchhoff vertex carrying no off-path attachment is treated
only as an internal point of the one-dimensional ODE, and each selected total
segment length is distributed arbitrarily into positive original edge lengths.

At every branching path vertex, the function vanishes on all incident edges,
so continuity holds; the two path derivatives cancel by construction and all
off-path derivatives are zero, and hence,the Kirchhoff condition holds. The
prescribed condition holds at \(a\) and \(b\), while on every active edge
\(-\phi''=\phi\). Therefore, \(\phi_\gamma^\star\) is an exact eigenfunction
at eigenvalue \(1=k_\star^2\), supported precisely on \(\gamma\).
\end{proof}

\begin{lemma}[Uniqueness of the path-supported eigendirection]
\label{lem:path-supported-unique}
For an auxiliary path scar as in \cref{prop:path-model-scar},
\begin{equation}
\label{eq:path-supported-eigenspace}
\mathcal E_\star^\gamma
:=
\left\{
f\in\ker(-\Delta_{\G,\ell^\star}-k_\star^2):
\supp f\subseteq\gamma
\right\}
=
\operatorname{span}\{\phi_\gamma^\star\}.
\end{equation}
\end{lemma}

\begin{proof}
Let \(f\in\mathcal E_\star^\gamma\), with
\(\mathcal E_\star^\gamma\) as in \eqref{eq:path-supported-eigenspace}.
At every branching path vertex, continuity with the inactive off-path edges
forces \(f=0\). Fix an oriented terminal segment and let \(v_0\) denote its
exterior endpoint. The map
\[
T:\mathcal E_\star^\gamma\longrightarrow\C,
\qquad
T(f)=
\begin{cases}
\partial_\nu f(v_0),& B_{v_0}=D,\\
f(v_0),& B_{v_0}=N,
\end{cases}
\]
is injective. Indeed, if \(T(f)=0\), then the prescribed exterior boundary
condition supplies the complementary zero Cauchy datum at \(v_0\). Uniqueness
for \(-f''=k_\star^2f\) gives \(f\equiv0\) on the first segment. At the next
branching vertex, continuity and Kirchhoff matching then give zero Cauchy data
for the following segment. Propagating along the path yields \(f\equiv0\) on
all of \(\gamma\). Hence,
$
\dim\mathcal E_\star^\gamma\le1.
$
The space is nonzero because it contains \(\phi_\gamma^\star\), proving the claim.

If \(\gamma\) has no branching vertex, connectedness implies that the graph
itself is an interval after degree-two subdivisions. The separated \(DD\),
\(DN\), or \(NN\) Sturm--Liouville spectrum is simple, and the same conclusion
follows directly.
\end{proof}

\subsubsection{Cycle scars}

\begin{proposition}[Cycle model scar]
\label{prop:cycle-model-scar}
Let \(C\subset\G\) be a simple cycle.
If \(C\neq\G\), there exist positive auxiliary edge lengths, a frequency
\(k_\star>0\), and a nonzero eigenfunction \(\phi_C^\star\) such that
\[
-\Delta_{\G,\ell^\star}\phi_C^\star
=
k_\star^2\phi_C^\star,
\qquad
\supp\phi_C^\star=C.
\]
Moreover, the eigendirections at \(k_\star^2\) supported entirely on \(C\)
form a one-dimensional space.
If \(C=\G\), the cycle already exhausts the graph and no off-support
regularization is required for the observability obstruction.
\end{proposition}

\begin{proof}
Assume first that \(C\neq\G\).  Since \(\G\) is connected, \(C\) contains at
least one vertex incident to an edge outside \(C\).  Subdividing only for the
local ODE description, if necessary, we regard every such attachment point as
a vertex of \(C\).  Let \(v_1,\ldots,v_m\) denote the attachment vertices in
cyclic order.  We prescribe the cycle state to vanish at these vertices.
Choose the auxiliary lengths of the successive cycle arcs so that
$$
    k_\star L_j=n_j\pi,
    \qquad n_j\in\mathbb N,
$$
and choose the integers \(n_j\) so that \(\sum_{j=1}^m n_j\) is even.  On the
\(j\)-th arc take a sine solution vanishing at its two endpoints.  If \(A_j\)
is its amplitude, the Kirchhoff condition at the next attachment vertex gives
$$
    A_{j+1}=(-1)^{n_j}A_j.
$$
The parity condition on \(\sum_j n_j\) makes this recursion compatible after
one circuit of the cycle.  Taking \(A_1\neq0\), therefore, produces a nonzero
cycle state satisfying the Kirchhoff derivative condition at every attachment
vertex.  Set the function identically equal to zero on \(\G\setminus C\).
The off-cycle traces and derivatives then vanish, so continuity and the
Kirchhoff conditions hold at every attachment vertex.  The resulting function
\(\phi_C^\star\) is consequently an exact eigenfunction at \(k_\star^2\)
supported on \(C\).

It remains to determine the dimension of the eigenspace at \(k_\star^2\)
supported entirely on \(C\).  Let \(\psi\) be any such eigenfunction.  At each
attachment vertex \(v_j\), an incident off-cycle edge carries the identically
zero restriction of \(\psi\).  Continuity therefore forces
\(\psi(v_j)=0\).  Fix one attachment vertex \(v_1\) and cut the cycle there.
Since the value at the cut is zero, a cycle-supported solution is determined
by one outgoing derivative at \(v_1\).  The edge equation, together with
continuity and the Kirchhoff condition at successive cycle vertices,
determines the solution uniquely around the entire cycle.  Hence, the space of
cycle-supported eigendirections at \(k_\star^2\) has dimension at most one.
Since the construction above provides a nonzero such eigendirection, this
space is exactly one-dimensional.

The same argument covers a one-edge loop and a two-edge cycle formed by
parallel edges.  Cutting at an attachment vertex reduces the active support to
an interval, with the endpoint conditions encoding the continuity and
Kirchhoff conditions at the original cut vertex.
If \(C=\G\), every eigenfunction is supported on the whole graph, so no
off-support concentration argument is required.  This includes the pure-loop
and parallel-edge cycle cases.  They will be treated separately whenever
simplicity regularization is invoked below.
\end{proof}

\section{Scar-preserving regularization}
\label{sec:regularization}

The model scar need not initially correspond to a simple ambient eigenvalue.
We now show that multiplicity can be removed by varying only off-support
edge lengths while leaving the scar itself unchanged.

\begin{proposition}[Scar-preserving simplicity regularization]
\label{prop:simplicity}
Let \(K\subsetneq\G\) be either a path scar from
\cref{prop:path-model-scar} or a cycle scar from
\cref{prop:cycle-model-scar}. Assume that the space of
\(k_\star^2\)-eigenfunctions supported entirely on \(K\) is one-dimensional,
spanned by \(\phi_K^\star\). Then, the auxiliary off-support lengths may be
perturbed arbitrarily slightly so that \(k_\star^2\) remains an eigenvalue with
eigenfunction \(\phi_K^\star\) and becomes simple.
\end{proposition}

\begin{proof}
Let the ambient eigenspace at the scar eigenvalue be
\begin{equation}
\label{eq:regularization-eigenspace}
E_\star
=
\ker(-\Delta_{\G,\ell^\star}-k_\star^2).
\end{equation}
The eigenspace \eqref{eq:regularization-eigenspace} may have dimension greater
than one. Consider the one-parameter off-support deformation
\begin{equation}
\label{eq:off-support-deformation}
\ell_e(t)=
\begin{cases}
\ell_e^\star, & e\subset K,\\
(1-t)\ell_e^\star, & e\not\subset K,
\end{cases}
\qquad |t|\ll1.
\end{equation}
The deformation \eqref{eq:off-support-deformation} leaves every length on
\(K\) unchanged. Since \(\phi_K^\star\) vanishes identically on
\(\G\setminus K\), its restriction to \(K\), extended by zero to the deformed
off-support edges, continues to satisfy the same edge equations and vertex
conditions. Thus
\begin{equation}
\label{eq:preserved-scar-branch}
-\Delta_{\G,\ell(t)}\phi_K^\star
=
k_\star^2\phi_K^\star
\end{equation}
for every sufficiently small \(t\), where the same notation is used for this
natural continuation.

To analyze the remaining branches, put each edge on the fixed reference
interval \([0,1]\) by \(x=\ell_e(t)y\). Let \(\mathcal H\) be the fixed form
domain consisting of edgewise \(H^1(0,1)\) functions satisfying the vertex
continuity conditions and the prescribed exterior Dirichlet conditions.
Kirchhoff and exterior Neumann conditions are natural operator conditions and
do not alter \(\mathcal H\). In these coordinates, the eigenvalue problem is
\begin{equation}
\label{eq:generalized-eigenproblem}
a_t(u,v)=\lambda(t)m_t(u,v),
\qquad v\in\mathcal H,
\end{equation}
where
\begin{align}
\label{eq:regularization-forms}
a_t(u,v)
&=
\sum_{e\in\E}\frac{1}{\ell_e(t)}
\int_0^1 u'_e\overline{v'_e}\,\dd y,
\\
m_t(u,v)
&=
\sum_{e\in\E}\ell_e(t)
\int_0^1 u_e\overline{v_e}\,\dd y.
\end{align}
The forms in \eqref{eq:regularization-forms} depend analytically on \(t\), the
form \(m_t\) is uniformly positive definite for \(|t|\) sufficiently small,
and the associated generalized problem has a compact resolvent. After the
standard analytic conjugation by the positive form \(m_t\), this becomes an
analytic self-adjoint family on a fixed Hilbert space. Degenerate analytic
perturbation theory~\cite{Kato1995}, therefore, shows that the first derivatives
of the eigenvalue branches issuing from \(k_\star^2\) are the generalized
eigenvalues, relative to \(m_0\), of the compressed Hermitian form
\begin{equation}
\label{eq:compressed-variation-form}
Q_{\mathrm{off}}(f,g)
:=
a'_0(f,g)-k_\star^2m'_0(f,g),
\qquad f,g\in E_\star,
\end{equation}
where \(E_\star\) is identified with its fixed-coordinate representatives.\\
Since \(\ell'_e(0)=0\) on \(K\) and
\(\ell'_e(0)=-\ell_e^\star\) off \(K\), differentiating
\eqref{eq:regularization-forms} gives
\begin{equation}
\label{eq:off-support-variation-reference}
Q_{\mathrm{off}}(f,g)
=
\sum_{e\not\subset K}
\left[
\frac{1}{\ell_e^\star}
\int_0^1 f'_e\overline{g'_e}\,\dd y
+
k_\star^2\ell_e^\star
\int_0^1 f_e\overline{g_e}\,\dd y
\right].
\end{equation}
Equivalently, in physical coordinates at \(t=0\),
\begin{equation}
\label{eq:off-support-variation-form}
Q_{\mathrm{off}}(f,g)
=
\sum_{e\not\subset K}
\int_e
\left(
\partial_x f_e\,\overline{\partial_x g_e}
+
k_\star^2 f_e\overline{g_e}
\right)\,\dd x.
\end{equation}
It follows from \eqref{eq:off-support-variation-form} that
\(Q_{\mathrm{off}}(f,f)\ge0\), with equality if and only if \(f\) vanishes
identically on every edge of \(\G\setminus K\). By the one-dimensional
supported-state hypothesis,
\begin{equation}
\label{eq:variation-kernel}
\ker Q_{\mathrm{off}}
=
\{f\in E_\star:\supp f\subseteq K\}
=
\operatorname{span}\{\phi_K^\star\}.
\end{equation}
Since \(m_0\) is positive definite on \(E_\star\) and
\(Q_{\mathrm{off}}\) is positive semidefinite with the one-dimensional kernel
in \eqref{eq:variation-kernel}, the generalized eigenvalue problem for
\((Q_{\mathrm{off}},m_0)\) has exactly one zero eigenvalue, while all remaining
generalized eigenvalues are strictly positive. Hence, exactly one eigenvalue
branch issuing from \(k_\star^2\) has zero first derivative. By
\eqref{eq:preserved-scar-branch}, this is the preserved scar branch,
\[
\lambda_{\mathrm{scar}}(t)\equiv k_\star^2.
\]
After analytic labeling, every other branch issuing from \(k_\star^2\)
satisfies
\begin{equation}
\label{eq:regularization-moving-branches}
\lambda_j(t)
=
k_\star^2+c_jt+o(t),
\qquad c_j>0.
\end{equation}
Because \(k_\star^2\) is an isolated eigenvalue at \(t=0\), spectral
continuity provides a neighborhood of \(k_\star^2\) containing, for all
sufficiently small \(t\), only the branches issuing from this eigenvalue.
Equation \eqref{eq:regularization-moving-branches} then shows that for every
sufficiently small \(t>0\), the preserved eigenvalue \(k_\star^2\) is simple.
Since \(t\) may be chosen arbitrarily small, the required perturbation may be
made arbitrarily small.
\end{proof}

\begin{remark}
\label{rem:regularization-scope}
The exterior Dirichlet--Neumann assignment is fixed during the deformation.
No boundary parameter is varied. Hence, mixed exterior conditions introduce
no additional perturbative term.
\end{remark}

\section{Secular geometry and arithmetic scar accessibility}
\label{sec:secular-transfer}

We now transfer a regular auxiliary scar to exact eigensequences of a fixed
target metric.  The target metric need not be rationally independent.

\subsection{Unitary secular formulation}

Orient every edge in both directions and let \(\mathcal B\) denote the set of
directed bonds. For a Kirchhoff vertex \(v\) of degree \(d_v\), the local
scattering matrix is
\begin{equation}
\label{eq:kirchhoff-scattering}
S_v=\frac{2}{d_v}J_{d_v}-I_{d_v}.
\end{equation}
At an exterior Dirichlet or Neumann vertex, the scalar reflection coefficient
is \(S_D=-1\) or \(S_N=+1\), respectively.
The local rule \eqref{eq:kirchhoff-scattering}, together with the exterior reflection coefficients, assembles into a unitary scattering matrix \(S_B\) that is
independent of the frequency \(k\).

Set \(\T=\R/(2\pi\Z)\). For
\(\theta=(\theta_e)_{e\in\E}\in\T^\E\), let \(D(\theta)\) denote propagation by the phase \(e^{i\theta_e}\) on both
directed bonds associated with \(e\), and define the secular unitary
\begin{equation}
\label{eq:secular-unitary}
U_B(\theta)=D(\theta)S_B.
\end{equation}
For a directed-bond vector \(a\in\C^{\mathcal B}\), write
\begin{equation}
\label{eq:bond-edge-intensity}
I_e(a):=|a_{e,+}|^2+|a_{e,-}|^2,
\qquad e\in\E,
\end{equation}
The intensity in \eqref{eq:bond-edge-intensity} will be used both in the Gauss identity and in physical mass estimates.  Here, \(a_{e,+}\) and \(a_{e,-}\) are the two directed-bond coordinates
associated with the unoriented edge \(e\).  This quantity is independent of
the choice of which orientation is denoted by \(+\) or \(-\).
The associated secular set is
\begin{equation}
\label{eq:secular-set}
Z_B=
\left\{
\theta\in\T^\E:
\det(I-U_B(\theta))=0
\right\}.
\end{equation}
By \eqref{eq:secular-unitary}--\eqref{eq:secular-set}, the target metric \(\ell\) satisfies
\begin{equation}
\label{eq:spectral-secular-condition}
k^2\in\spec(-\Delta_{\G,\ell})
\quad\Longleftrightarrow\quad
[k\ell]\in Z_B.
\end{equation}

\begin{lemma}[Spectral--secular multiplicity correspondence]
\label{lem:secular-multiplicity}
Let \(k>0\) and \(\ell\in(0,\infty)^\E\). Then reconstruction from directed
bond amplitudes defines a linear isomorphism
\begin{equation}
\label{eq:spectral-secular-isomorphism}
\ker\bigl(I-U_B(k\ell)\bigr)
\cong
\ker\bigl(-\Delta_{\G,\ell}-k^2\bigr).
\end{equation}
In particular, the two spaces in
\eqref{eq:spectral-secular-isomorphism} have the same dimension. Hence
\(k^2\) is a simple Laplacian eigenvalue if and only if \(1\) is a simple
eigenvalue of \(U_B(k\ell)\).
\end{lemma}

\begin{proof}
Let \(u\) be an eigenfunction at energy \(k^2\). On each edge, write \(u\) in
traveling-wave form and collect the incoming and outgoing coefficients into
the directed-bond amplitude vectors \(a^{\rm in}\) and \(a^{\rm out}\).
The vertex conditions and edge propagation give
\begin{equation}
\label{eq:bond-amplitude-relations}
a^{\rm out}=S_Ba^{\rm in},
\qquad
a^{\rm in}=D(k\ell)a^{\rm out}.
\end{equation}
With convention
$$
U_B(k\ell)=D(k\ell)S_B,
$$
the relations \eqref{eq:bond-amplitude-relations} imply
$$
a^{\rm in}=U_B(k\ell)a^{\rm in}.
$$
Thus, the traveling-wave coefficients define a linear map

$$
\ker\bigl(-\Delta_{\G,\ell}-k^2\bigr)
\longrightarrow
\ker\bigl(I-U_B(k\ell)\bigr).
$$
Conversely, let
\(a^{\rm in}\in\ker(I-U_B(k\ell))\) and set
\(a^{\rm out}=S_Ba^{\rm in}\). The fixed-point relation then gives
$$
a^{\rm in}=D(k\ell)a^{\rm out}.
$$
Using these incoming and outgoing coefficients on each edge reconstructs a
solution of \(-u''=k^2u\). The scattering relation enforces the prescribed
Kirchhoff or exterior Dirichlet--Neumann vertex conditions, while the
propagation relation identifies the amplitudes at the two ends of each edge.
Hence, the reconstructed function belongs to
\(\ker(-\Delta_{\G,\ell}-k^2)\).

Since \(k>0\), the two traveling waves on each edge are linearly independent.
The traveling-wave coefficients of an edge solution are therefore unique.
Consequently, the two constructions above are inverse linear maps, proving
the isomorphism \eqref{eq:spectral-secular-isomorphism}.
Finally, \(U_B(k\ell)\) is unitary, so the eigenvalue \(1\) is semisimple and
its algebraic and geometric multiplicities coincide. Therefore,
$$
\dim\ker\bigl(I-U_B(k\ell)\bigr)
=
\dim\ker\bigl(-\Delta_{\G,\ell}-k^2\bigr),
$$
which proves the multiplicity statement.
\end{proof}

\begin{lemma}[Mixed-boundary Gauss identity]
\label{lem:gauss}
Let \(\theta^\star\in Z_B\) be a point at which \(1\) is a simple eigenvalue of
\(U_B(\theta^\star)\).
Then, in a neighborhood of \(\theta^\star\), there are a smooth normalized
eigenvector \(a(\theta)\) and a smooth eigenphase \(\varphi(\theta)\) such that
\begin{equation}
\label{eq:local-secular-eigenpair}
U_B(\theta)a(\theta)
=
e^{i\varphi(\theta)}a(\theta),
\qquad
\varphi(\theta^\star)=0.
\end{equation}
Moreover, for every edge \(e\), the normal derivative is given by
\begin{equation}
\label{eq:gauss-intensity}
\frac{\partial\varphi}{\partial\theta_e}(\theta)
=
I_e(a(\theta))
=
|a_{e,+}(\theta)|^2+|a_{e,-}(\theta)|^2.
\end{equation}
\end{lemma}

\begin{proof}
Since \(1\) is a simple eigenvalue, standard perturbation theory for finite
dimensional unitary matrices yields the local smooth eigenpair.
Differentiate \eqref{eq:local-secular-eigenpair} with respect to
\(\theta_e\), take the inner product with \(a\), and use
\(\|a\|=1\). The unitary Hellmann--Feynman identity gives
\[
\partial_{\theta_e}\varphi
=
-i\,a^\ast U_B^\ast
(\partial_{\theta_e}U_B)a.
\]
Let \(P_e\) be the orthogonal projector onto the two directed-bond coordinates
associated with \(e\). Since \(\partial_{\theta_e}D=iP_eD\),
\[
-i\,U_B^\ast(\partial_{\theta_e}U_B)=S_B^\ast P_eS_B.
\]
At every point of the chosen local eigenbranch,
\(D(\theta)S_Ba=e^{i\varphi}a\),
so \(S_Ba=e^{i\varphi}D(\theta)^\ast a\). The diagonal unitary
\(D(\theta)^\ast\) preserves componentwise moduli and commutes with \(P_e\).
Hence,
\[
\partial_{\theta_e}\varphi
=\langle S_Ba,P_eS_Ba\rangle
=\langle a,P_ea\rangle
=|a_{e,+}|^2+|a_{e,-}|^2.
\]
The signs \(S_D=-1\) and \(S_N=+1\) are contained in the fixed unitary matrix
\(S_B\) and therefore do not alter the derivative identity.
\end{proof}

\begin{corollary}[Positive transversality]
\label{cor:transversality}
Let \(\theta^\star\in Z_B\) be a simple secular point whose corresponding
eigendirection is supported on a nonempty subgraph \(K\subset\G\), and let
\(\ell\in(0,\infty)^\E\) be any positive target metric.
Then
\begin{equation}
\label{eq:positive-transversality}
D_\ell\varphi(\theta^\star)
=
\ell\cdot\nabla\varphi(\theta^\star)
=
\sum_{e\subset K}
\ell_e
\left(
|a_{e,+}^\star|^2+|a_{e,-}^\star|^2
\right)
>0.
\end{equation}
\end{corollary}

\begin{proof}
For a scar supported on \(K\), the bond amplitudes vanish on every edge outside
\(K\). By \Cref{lem:gauss} and \eqref{eq:gauss-intensity}, every summand in \eqref{eq:positive-transversality} is
nonnegative. At least one active edge has nonzero bond amplitude, and every
target length is strictly positive. Hence, the sum is strictly positive.
\end{proof}

\begin{definition}[Regular scar stratum]
\label{def:regular-scar-stratum}
For a proper primitive subgraph \(K\), let \(\mathcal R_K\subset Z_B\) denote
the set of secular phases \(\theta\) for which \(1\) is a simple eigenvalue of
\(U_B(\theta)\) and the corresponding eigendirection is supported exactly on
\(K\).
\end{definition}

\begin{corollary}[Nonemptiness of the regular scar stratum]
\label{cor:regular-scar-nonempty}
For every proper simple cycle and every proper simple exterior-to-exterior path
\(K\), with any fixed endpoint type in the path case,
\begin{equation}
\label{eq:regular-scar-nonempty}
\mathcal R_K\neq\varnothing.
\end{equation}
\end{corollary}

Equation \eqref{eq:regular-scar-nonempty} is the regularity input used in the accessibility theory below.

\begin{proof}
The model-scar results in \Cref{sec:model-scars} provide a \(K\)-supported
auxiliary eigenstate and a one-dimensional \(K\)-supported eigendirection.
\Cref{prop:simplicity} makes the ambient eigenvalue simple by an arbitrarily
small off-support deformation while preserving the scar.  The
spectral--secular multiplicity correspondence in
\cref{lem:secular-multiplicity} then identifies the resulting phase as a point
of \(\mathcal R_K\).
\end{proof}

\subsection{Relative recurrent crossings on the orbit-closure torus}

\begin{theorem}[Relative recurrent crossing]
\label{thm:relative-crossing}
Let \(\ell\in(0,\infty)^\E\), and let \(\theta^\star\in Z_B\) be a regular
simple secular point for which
\(D_\ell\varphi(\theta^\star)\neq0\). Then, the following are equivalent:
\begin{enumerate}[label=(\roman*)]
\item \(\theta^\star\in H_\ell\);
\item there exist \(k_n\to\infty\) such that
\begin{equation}
\label{eq:relative-crossings}
[k_n\ell]\in Z_B,
\qquad
[k_n\ell]\longrightarrow\theta^\star.
\end{equation}
\end{enumerate}
\end{theorem}

\begin{proof}
If \eqref{eq:relative-crossings} holds, then
\(\theta^\star\in H_\ell\) by the definition of the orbit closure.
Conversely, assume that \(\theta^\star\in H_\ell\). Since \(H_\ell\) is a
compact subtorus invariant under the flow
\[
\Phi_t(\theta)=\theta+t\ell,
\]
the vector \(\ell\) is tangent to \(H_\ell\). The hypothesis
\(D_\ell\varphi(\theta^\star)\neq0\) implies that the restriction of
\(\varphi\) to \(H_\ell\) has nonzero differential at \(\theta^\star\).
Hence, by the implicit-function theorem on \(H_\ell\),
\(\{\varphi=0\}\cap H_\ell\) is a smooth codimension-one submanifold near
\(\theta^\star\), transverse to the flow generated by \(\ell\).

After shrinking the neighborhood, the relative flow-box theorem yields a
relatively open disk
\[
\Sigma_0\subset\{\varphi=0\}\cap H_\ell
\subset Z_B\cap H_\ell,
\]
a number \(\varepsilon>0\), and a diffeomorphism
\begin{equation}
\label{eq:relative-flow-box}
\Psi:(-\varepsilon,\varepsilon)\times\Sigma_0
\longrightarrow U_0\subset H_\ell,
\qquad
\Psi(t,\sigma)=\sigma+t\ell.
\end{equation}
Choose a relatively open disk
\(\Sigma_1\Subset\Sigma_0\) containing \(\theta^\star\), and define the
incoming slab
\[
U_-=
\Psi\bigl((-\varepsilon,-\varepsilon/2)\times\Sigma_1\bigr).
\]
The forward orbit \(\{[t\ell]:t\ge0\}\) is dense in \(H_\ell\). Indeed, its
closure is a compact subsemigroup of the compact group \(H_\ell\), hence, a
subgroup, and therefore coincides with the closure of the full one-parameter
group, namely \(H_\ell\). The same conclusion holds for every forward tail.
Consequently, for every \(T>0\) there exists \(s>T\) such that
$
[s\ell]\in U_-.
$
Write
$
[s\ell]=\Psi(-\tau,\sigma)
$
with
\(\tau\in(\varepsilon/2,\varepsilon)\) and
\(\sigma\in\Sigma_1\). By \eqref{eq:relative-flow-box},
\[
[(s+\tau)\ell]
=
\Psi(0,\sigma)
=
\sigma
\in Z_B\cap H_\ell.
\]
Thus, every sufficiently small relative flow box around
\(\theta^\star\) produces an exact secular crossing at an arbitrarily large
target frequency.
Choose nested flow boxes for which the corresponding disks
\(\Sigma_1^{(n)}\) shrink to \(\theta^\star\), and choose the successive entry
times so large that the resulting crossing times \(k_n\) satisfy
\(k_n\to\infty\). The corresponding crossing points belong to
\(\Sigma_1^{(n)}\), and hence,
\[
[k_n\ell]\longrightarrow\theta^\star.
\]
This proves \eqref{eq:relative-crossings}.
\end{proof}
By \cref{lem:orbit-closure-lattice}, the recurrence condition in
\cref{thm:relative-crossing} has the arithmetic form
\begin{equation}
\label{eq:scar-resonance-compatibility}
m\cdot\theta^\star\in2\pi\Z
\qquad\text{for every }m\in\Lambda_\ell.
\end{equation}
The condition \eqref{eq:scar-resonance-compatibility} makes the role of target resonances explicit.

\begin{definition}[Regular arithmetic accessibility]
\label{def:RAC}
For a proper primitive subgraph \(K\), define
\begin{equation}
\label{eq:RAC}
\operatorname{RAC}(K,\ell)
\quad\text{to mean}\quad
H_\ell\cap\mathcal R_K\neq\varnothing.
\end{equation}
\end{definition}
The condition in \eqref{eq:RAC} will be abbreviated as RAC below.

\subsection{Arithmetic accessibility of primitive scars}

\begin{theorem}[Arithmetic primitive-scar accessibility]
\label{thm:arithmetic-scar-accessibility}
Let \(K\subsetneq\G\) be a simple cycle or a simple path joining two exterior
vertices.  If \(\operatorname{RAC}(K,\ell)\) holds, then there exist normalized
exact eigenfunctions satisfying
\begin{equation}
\label{eq:arithmetic-scar-eigenfunctions}
-\Delta_{\G,\ell}u_n=k_n^2u_n,
\qquad
k_n\to\infty,
\qquad
\|u_n\|_{L^2(\G\setminus K)}\longrightarrow0,
\end{equation}
and
\begin{equation}
\label{eq:arithmetic-scar-measure}
|u_n|^2\dd x
\rightharpoonup
\frac1{L_K}\,\dd x\big|_K.
\end{equation}
All four exterior-path endpoint types \(DD,DN,ND,NN\) are allowed.
\end{theorem}

\begin{proof}
Choose \(\theta^\star\in H_\ell\cap\mathcal R_K\).  By
\cref{cor:transversality}, \(D_\ell\varphi(\theta^\star)>0\), and
\cref{thm:relative-crossing} gives exact target spectral points
\([k_n\ell]\to\theta^\star\).  On a sufficiently small neighborhood the
simple local eigenbranch in \eqref{eq:local-secular-eigenpair} persists.  Choose unit eigenvectors \(a_n\) on that branch.  After a consistent phase choice,
$
a_n\longrightarrow a^\star.
$
Here, \(a^\star\) is the \(K\)-supported eigenvector at \(\theta^\star\).

Let \(w_n\) be the physical eigenfunction reconstructed from \(a_n\), and let
\(u_n=c_nw_n\) be its \(L^2\)-normalization.  Since the two traveling-wave
coefficients on an edge have squared moduli summing to \(I_e(a_n)\), exact edge
integration and convergence \(a_n\to a^\star\) show
\begin{equation}
\label{eq:physical-normalization-arithmetic}
\|w_n\|_{L^2(\G)}^2
\longrightarrow
M_\star:=\sum_{e\in\E}\ell_e I_e(a^\star)>0.
\end{equation}
By \eqref{eq:physical-normalization-arithmetic}, the normalization factors remain bounded and bounded away from zero.
For \(e\not\subset K\), \(I_e(a^\star)=0\), so the reconstructed mass on every
off-support edge tends to zero.  This gives
\eqref{eq:arithmetic-scar-eigenfunctions}.  Since \(\G\) is compact, every
subsequence of the probability measures \(|u_n|^2\dd x\) has a further weakly
convergent subsequence.  The off-support concentration forces each such limit to have support contained
in \(K\).  Since the limit is a probability measure, its support is nonempty; a
proper nonempty union of original edges contained in the primitive path or cycle
\(K\) would have an interior leaf, contrary to \cref{lem:support-law}.
Hence, the support is exactly \(K\), and
\cref{cor:primitive-measure-unique} identifies the limit with normalized
arclength.  Thus, all subsequential limits coincide, proving the full convergence
in \eqref{eq:arithmetic-scar-measure}.
\end{proof}

\begin{corollary}[Rationally independent accessibility]
\label{cor:ri-accessibility}
If \(\ell\) is rationally independent, then
\(\operatorname{RAC}(K,\ell)\) holds for every proper primitive \(K\), and
\cref{thm:arithmetic-scar-accessibility} applies to each such \(K\).
\end{corollary}

\begin{proof}
Rational independence gives \(H_\ell=\T^\E\), while
\cref{cor:regular-scar-nonempty} gives \(\mathcal R_K\neq\varnothing\).
\end{proof}

\begin{definition}[Graph-wide regular accessibility class]
\label{def:AG}
Define
\begin{equation}
\label{eq:AG}
\mathcal A_\G
:=
\left\{\ell\in(0,\infty)^\E:
\operatorname{RAC}(K,\ell)
\text{ for every proper primitive }K\subset\G
\right\}.
\end{equation}
\end{definition}
The class in \eqref{eq:AG} contains every rationally independent metric by \cref{cor:ri-accessibility}.

\begin{theorem}[Mixed-boundary minimal-scar characterization]
\label{thm:minimal-scar-characterization}
Let \(\ell\in\mathcal A_\G\).  The minimal supports of fixed-metric probability
semiclassical measures are precisely the simple cycles and the simple paths
joining two exterior vertices.  On every such support \(K\), the probability
semiclassical measure is uniquely
\(L_K^{-1}\dd x|_K\).  In particular, the conclusion holds for every
rationally independent metric.
\end{theorem}

\begin{proof}
Let \(H\) be a minimal support.  By \cref{lem:support-law}, the finite subgraph
\(H\) has no leaf at an interior vertex.  Therefore, \(H\) contains either a
simple cycle or a simple path joining two exterior vertices; call it \(K\).
If \(K\subsetneq\G\), membership \(\ell\in\mathcal A_\G\) and
\cref{thm:arithmetic-scar-accessibility} realize \(K\) as the support of a
semiclassical measure for the same target metric.  Since \(K\subseteq H\),
minimality of \(H\) forces \(H=K\).  If \(K=\G\), then \(H=K\) is
immediate.

Conversely, let \(K\) be a primitive cycle or exterior-to-exterior path.  If
\(K\subsetneq\G\), its realization follows from
\cref{thm:arithmetic-scar-accessibility}.  If \(K=\G\), then the graph is a
circle or an interval up to retained degree-two subdivisions; any sequence of
normalized eigenfunctions with frequencies tending to infinity has, after choosing
a standard sine/cosine branch, the normalized-arclength weak limit by direct
one-dimensional oscillatory averaging.  Thus, \(K\) is the support of a
semiclassical measure in all cases.  Any proper nonempty support contained in a
primitive cycle or exterior-to-exterior path has an interior leaf, because
\cref{lem:edgewise-flatness} makes supports unions of whole original edges.
\Cref{lem:support-law} therefore excludes a smaller semiclassical support.
Uniqueness of the probability measure on \(K\) follows from
\cref{cor:primitive-measure-unique}.  The final statement follows from
\cref{cor:ri-accessibility}.
\end{proof}

\begin{remark}[Scope of the arithmetic criterion]
\label{rem:RAC-scope}
The condition \(\operatorname{RAC}(K,\ell)\) characterizes accessibility through
the regular-scar mechanism used here.  We do not claim that
\(\operatorname{RAC}(K,\ell)\) is necessary for every possible scar sequence:
a singular secular stratum could, in principle, provide a different
accessibility mechanism.
\end{remark}

\section{Quantitative Diophantine scar accessibility}
\label{sec:quantitative}

We now quantify the regular-scar mechanism.  Put
\begin{equation}
\label{eq:intrinsic-dimension}
r_\ell:=\dim H_\ell.
\end{equation}
With \(r_\ell\) fixed by \eqref{eq:intrinsic-dimension}, choose a Lie-group isomorphism induced by integral lattice coordinates
\(\Xi:\T^{r_\ell}\to H_\ell\) and an intrinsic flow direction
\(\beta\in\R^{r_\ell}\) satisfying
\(\Xi([t\beta])=[t\ell]\).  With fixed flat metrics, \(\Xi\) and its
inverse are bi-Lipschitz.  Thus, changing integral coordinates changes the constants
below but not the exponents or the substance of the condition.

\begin{definition}[Intrinsic Diophantine condition]
\label{def:intrinsic-diophantine}
Assume \(r_\ell\ge2\).  Put \(\widehat\beta:=\beta/\|\beta\|\).  We say that the intrinsic flow is \((\gamma,\tau)\)-Diophantine if
\begin{equation}
\label{eq:intrinsic-diophantine}
|q\cdot\widehat\beta|
\ge\gamma\|q\|^{-\tau}
\qquad
\text{for every }q\in\Z^{r_\ell}\setminus\{0\},
\qquad
\tau>r_\ell-1.
\end{equation}
\end{definition}
The filling-time estimate of Dumas--Fischler~\cite{DumasFischler2022} implies
that a \((\gamma,\tau)\)-Diophantine linear flow becomes \(\rho\)-dense in the
intrinsic torus within time at most \(C\rho^{-\tau}\), with \(C\) depending on
the intrinsic dimension and Diophantine constants.

\begin{theorem}[Quantitative recurrent crossing]
\label{thm:quantitative-crossing}
Let \(\theta^\star\in H_\ell\cap\mathcal R_K\).  If \(r_\ell\ge2\) and the
intrinsic target flow satisfies \eqref{eq:intrinsic-diophantine}, then there are
exact target spectral points \(\theta_n=[k_n\ell]\in Z_B\) with \(k_n\to\infty\)
such that
\begin{equation}
\label{eq:quantitative-crossing}
\mathrm{dist}_{H_\ell}(\theta_n,\theta^\star)
\le C k_n^{-1/\tau}.
\end{equation}
If \(r_\ell=1\), the continuous flow on \(H_\ell\) is periodic and
\(\theta^\star\in H_\ell\) gives arbitrarily large exact returns to
\(\theta^\star\).
\end{theorem}

\begin{proof}
The case \(r_\ell=1\) is periodic.  Assume \(r_\ell\ge2\).  By positive
transversality, a relative flow box exists near \(\theta^\star\).  In its
coordinates \(\Psi(t,\sigma)=\sigma+t\ell\), choose for each sufficiently small
\(\rho>0\) a relative ball \(B^-_\rho\) of radius \(c\rho\) centered at a
point \(\Psi(-c_1\rho,\theta^\star)\).  Every orbit entering
\(B^-_\rho\) crosses the zero-time slice at a point whose intrinsic distance
from \(\theta^\star\) is \(O(\rho)\), and the additional crossing time is
\(O(\rho)\).  After transporting by \(\Xi^{-1}\), the Dumas--Fischler filling
estimate (with \(\gamma\) replaced by \(\min\{\gamma,1/2\}\) if necessary)
applied to the unit direction \(\widehat\beta\) gives an entry within intrinsic
unit-speed time \(C\rho^{-\tau}\), uniformly with respect to the starting
point.  Returning from \(\widehat\beta\) to the physical intrinsic direction
\(\beta\) only rescales time by the fixed factor \(\|\beta\|^{-1}\), which
is absorbed into the constant.

To relate the radius to the actual crossing frequency, start the filling
argument at a prescribed late time \(K_0=c_0\rho^{-\tau}\).  Translation
invariance of the flow gives an entry during
\([K_0,K_0+C\rho^{-\tau}]\).  The subsequent flow-box crossing changes the time
by a bounded amount.  Hence, the exact crossing frequency satisfies
\(c\rho^{-\tau}\le k_\rho\le C'\rho^{-\tau}\) after decreasing \(\rho\) if
necessary.  The crossing itself lies within distance \(C''\rho\) of
\(\theta^\star\).  Choosing \(\rho\downarrow0\) and eliminating \(\rho\) gives
\eqref{eq:quantitative-crossing}.
\end{proof}

\begin{lemma}[Quadratic local leakage]
\label{lem:quadratic-leakage}
Let \(\theta^\star\in\mathcal R_K\) and let \(\theta=[k\ell]\in Z_B\) be an
exact target spectral point on the corresponding local simple eigenbranch,
sufficiently close to \(\theta^\star\), with \(k\) sufficiently large.  If
\(u_\theta\) is the associated normalized physical eigenfunction, then
\begin{equation}
\label{eq:quadratic-leakage}
\|u_\theta\|_{L^2(\G\setminus K)}^2
\le
C\operatorname{dist}(\theta,\theta^\star)^2.
\end{equation}
\end{lemma}

\begin{proof}
Smoothness of the rank-one spectral projector for the simple branch in
\eqref{eq:local-secular-eigenpair} gives, after a consistent phase choice,
\(\|a(\theta)-a(\theta^\star)\|\le C r\), where
\(r=\operatorname{dist}(\theta,\theta^\star)\).  Since the reference scar has
zero bond amplitudes off \(K\),
\begin{equation}
\label{eq:quadratic-off-intensity}
I_e(a(\theta))\le C r^2,
\qquad e\not\subset K.
\end{equation}
On an edge with traveling-wave coefficients \(A_e,B_e\), exact integration
gives
\[
\int_0^{\ell_e}|A_ee^{ikx}+B_ee^{-ikx}|^2\dd x
=\ell_e I_e+R_e,
\qquad
|R_e|\le k^{-1}I_e.
\]
Thus, the off-support physical mass before normalization is bounded by
\(Cr^2\).  For unit bond vectors, the nonoscillatory total mass is
\(\sum_e\ell_e I_e\ge\min_e\ell_e\); the total oscillatory remainder is
\(O(k^{-1})\).  Hence, for large \(k\) the physical normalization factor is
uniformly bounded above and below.  Combining these bounds with
\eqref{eq:quadratic-off-intensity} proves \eqref{eq:quadratic-leakage}.
\end{proof}

\begin{theorem}[Quantitative scar accessibility]
\label{thm:quantitative-scar}
Let \(K\subsetneq\G\) satisfy \(\operatorname{RAC}(K,\ell)\), and choose
\(\theta^\star\in H_\ell\cap\mathcal R_K\).  If \(r_\ell\ge2\) and the
intrinsic target flow satisfies \eqref{eq:intrinsic-diophantine}, then there is
a normalized exact eigensequence with
\begin{equation}
\label{eq:quantitative-scar}
-\Delta_{\G,\ell}u_n=k_n^2u_n,
\qquad
\|u_n\|_{L^2(\G\setminus K)}^2
\le C k_n^{-2/\tau}.
\end{equation}
\end{theorem}

\begin{proof}
Apply \cref{thm:quantitative-crossing}.  The intrinsic distance on
\(H_\ell\) and the ambient phase-torus distance are locally comparable under the
fixed embedding, so \eqref{eq:quantitative-crossing} may be inserted into
\cref{lem:quadratic-leakage}.
\end{proof}

\begin{remark}[No quantitative rate from rational independence alone]
\label{rem:no-ri-rate}
Rational independence gives density of the phase orbit but no uniform
polynomial shrinking-target rate.  The estimate
\eqref{eq:quantitative-scar} therefore requires the intrinsic Diophantine
hypothesis; no polynomial rate is asserted under rational independence alone.
\end{remark}

\section{Observability degeneration and graph geometric control}
\label{sec:ggcc-necessity}

Let \(\omega\) be a nonempty open subset of the disjoint union of edge
interiors, and consider, for \(T>0\),
\begin{equation}
\label{eq:controlled-schrodinger}
\begin{cases}
i\partial_t y=-\Delta_{\G,\ell}y+\mathbf 1_\omega h,
&(t,x)\in(0,T)\times\G,\
y(0)=y_0\in L^2(\G).
\end{cases}
\end{equation}
By the Hilbert Uniqueness Method (HUM), exact controllability in time \(T\) is
equivalent to the existence of a constant \(C_T>0\) such that every solution
of the adjoint equation
$$
z(t)=e^{it\Delta_{\G,\ell}}z_0
$$
satisfies the observability estimate
\begin{equation}
\label{eq:schrodinger-observability}
|z_0|*{L^2(\G)}^2
\le
C_T\int_0^T|z(t)|*{L^2(\omega)}^2,\dd t.
\end{equation}

\begin{lemma}[Primitive obstruction to GGCC]
\label{lem:ggcc-obstruction}
If \((\G,\omega)\) fails GGCC, then there exists a simple cycle or a simple
path joining two exterior vertices, denoted by \(K\), such that
\begin{equation}
\label{eq:primitive-uncontrolled}
K\cap\omega=\varnothing.
\end{equation}
\end{lemma}

\begin{proof}
By \cite[Theorem~1.2]{AmmariDucaJolyLeBalch2025}, GGCC holds if and only if every
cycle and every exterior-to-exterior path meets \(\omega\).  Negation gives a
cycle or path disjoint from \(\omega\), and removing repetitions gives a simple
one.  Degree-two subdivisions remain part of the original metric graph.
\end{proof}

\begin{theorem}[Metric-specific necessity through regular accessibility]
\label{thm:metric-specific-necessity}
Assume GGCC fails.  If there is an uncontrolled proper primitive obstruction
\(K\) satisfying \(\operatorname{RAC}(K,\ell)\), then
\eqref{eq:controlled-schrodinger} is not exactly controllable in any time
\(T>0\).\end{theorem}

\begin{proof}
For the obstruction satisfying \eqref{eq:primitive-uncontrolled}, \cref{thm:arithmetic-scar-accessibility} gives
normalized eigenfunctions \(u_n\) with
\(\|u_n\|_{L^2(\G\setminus K)}\to0\).  Since
\(\omega\subset\G\setminus K\),
\(\|u_n\|_{L^2(\omega)}\to0\).  For fixed \(T>0\), set
\(z_n(t)=e^{-itk_n^2}u_n\).  Then
\begin{equation}
\label{eq:necessity-observation-integral}
\int_0^T\|z_n(t)\|_{L^2(\omega)}^2\dd t
=T\|u_n\|_{L^2(\omega)}^2
\longrightarrow0,
\end{equation}
while \(\|z_n(0)\|=1\).  Thus, \eqref{eq:necessity-observation-integral} contradicts
\eqref{eq:schrodinger-observability}.  Since \(T>0\) was arbitrary, exact controllability fails for every positive time.
\end{proof}

\subsection{Quantitative finite-frequency observation}

For \(\Lambda>0\), define the best truncated observation level
\begin{equation}
\label{eq:truncated-observation-level}
\mathfrak o_T(\Lambda)
:=
\inf_{\substack{
 z_0\in\mathbf 1_{[0,\Lambda]}(-\Delta_{\G,\ell})L^2(\G)\\
 \|z_0\|_{L^2(\G)}=1}}
\int_0^T
\|e^{it\Delta_{\G,\ell}}z_0\|_{L^2(\omega)}^2\dd t.
\end{equation}

\begin{corollary}[Polynomial degeneration of truncated observation]
\label{cor:quantitative-observation}
Assume GGCC fails and let \(K\subsetneq\G\) be an uncontrolled primitive
obstruction satisfying the hypotheses of \cref{thm:quantitative-scar}.  Then
along \(\Lambda_n=k_n^2\),
\begin{equation}
\label{eq:quantitative-observation}
\mathfrak o_T(\Lambda_n)
\le C T\Lambda_n^{-1/\tau}.
\end{equation}
Equivalently, whenever the inverse is interpreted as \(+\infty\) at zero,
\begin{equation}
\label{eq:inverse-observation-growth}
\mathfrak o_T(\Lambda_n)^{-1}
\ge c T^{-1}\Lambda_n^{1/\tau}.
\end{equation}
\end{corollary}

\begin{proof}
Use the normalized eigenfunction \(u_n\) from
\eqref{eq:quantitative-scar} as a competitor in
\eqref{eq:truncated-observation-level}.  Since \(K\cap\omega=\varnothing\),
\[
\|u_n\|_{L^2(\omega)}^2
\le\|u_n\|_{L^2(\G\setminus K)}^2
\le Ck_n^{-2/\tau}.
\]
Stationarity of an eigenfunction under the free Schr\"odinger evolution then
gives \eqref{eq:quantitative-observation}; inversion gives
\eqref{eq:inverse-observation-growth}.
\end{proof}

\subsection{Generic GGCC characterization}

\begin{theorem}[GGCC characterization for rationally independent metrics]
\label{thm:main-control}
Let \(\ell\in(0,\infty)^\E\) be rationally independent.  Then the following are
equivalent:
\begin{enumerate}[label=(\roman*)]
\item \((\G,\omega)\) satisfies GGCC;
\item for every \(T>0\), the free internally controlled Schr\"odinger equation
\eqref{eq:controlled-schrodinger} is exactly controllable in time \(T\).
\end{enumerate}
If GGCC fails, exact controllability fails for every \(T>0\).
\end{theorem}

\begin{proof}
If GGCC holds, small-time exact controllability follows from
\cite[Theorem~1.9]{AmmariDucaJolyLeBalch2025}, with zero potential.  If GGCC
fails, \cref{lem:ggcc-obstruction} gives an uncontrolled primitive \(K\).  Since \(\omega\neq\varnothing\) and \(K\cap\omega=\varnothing\), the obstruction is proper.  Rational independence and \cref{cor:ri-accessibility} give
\(\operatorname{RAC}(K,\ell)\), so \cref{thm:metric-specific-necessity} applies.
\end{proof}

\begin{corollary}[Generic metric characterization]
\label{cor:generic-control}
Fix the combinatorial graph, the exterior Dirichlet--Neumann assignment, and the
nonempty open control geometry \(\omega\).  For Lebesgue-almost every
\(\ell\in(0,\infty)^\E\), GGCC holds if and only if
\eqref{eq:controlled-schrodinger} is exactly controllable in every time
\(T>0\).
\end{corollary}

\begin{proof}
The rationally dependent metrics are
\begin{equation}
\label{eq:rational-dependence-hyperplanes}
\bigcup_{m\in\Z^\E\setminus\{0\}}
\{\ell\in(0,\infty)^\E:m\cdot\ell=0\},
\end{equation}
The set in \eqref{eq:rational-dependence-hyperplanes} is a countable union of proper hyperplane sections and hence a Lebesgue-null set.
Apply \cref{thm:main-control}.
\end{proof}

\begin{remark}[Exceptional arithmetic metrics]
\label{rem:exceptional-metrics}
The generic qualification is essential: arithmetic exceptional configurations
with failure of GGCC but exact controllability are exhibited in
\cite{AmmariDucaJolyLeBalch2025}.  The current theory refines the role of
arithmetic by showing that rational independence is sufficient but not
structurally necessary for the regular-scar obstruction: the relevant
metric-specific condition is intersection of \(H_\ell\) with the regular scar
stratum of an uncontrolled primitive support.  No complete classification of
all exceptional metrics is asserted because singular secular accessibility is
not analyzed.
\end{remark}

\section{Discussion and scope}
\label{sec:discussion}

The results distinguish three roles that are not visible in a purely generic
formulation.  The graph topology determines the primitive obstructions to GGCC;
the secular geometry determines the regular scar strata associated with those
obstructions; and the target metric determines which of these strata are
accessible through its phase-orbit closure.  Rational independence is therefore
a convenient sufficient condition ensuring full phase-orbit closure, rather
than the intrinsic metric-specific accessibility condition.  Under the stronger
intrinsic Diophantine hypothesis, the same mechanism yields the quantitative
observation degeneration in \eqref{eq:quantitative-observation}.

The scope of these conclusions is deliberately limited.  The condition
\(\operatorname{RAC}(K,\ell)\) concerns accessibility through the regular-scar
mechanism and is not asserted to characterize accessibility through singular
secular strata.  Likewise, rational independence alone is not used to infer a
polynomial recurrence rate, and no necessity result is asserted here for
nonzero bounded potentials.  The generic free-Schr\"odinger characterization
in \cref{thm:main-control} thus follows as a control-theoretic consequence of
the fixed-metric spectral theory developed above.



\bibliographystyle{plain}
\bibliography{refs}

\end{document}